\documentclass{amsart}
\usepackage{amstext,amssymb,amsthm,amsopn,newlfont,graphpap,graphics,graphicx,mathrsfs,enumitem}\usepackage{amstext,amssymb,amsthm,amsopn,newlfont,graphpap,graphics,graphicx,mathrsfs,enumitem,tikz-cd}
\usepackage[parfill]{parskip}
\usepackage[noadjust]{cite}
\usepackage{epigraph}
\usepackage[colorlinks=true,
            linkcolor=red,
            urlcolor=blue,
            citecolor=magenta]{hyperref}
\usepackage{color}
\usepackage{mathrsfs}
\allowdisplaybreaks
\theoremstyle{plain}
\newtheorem{thm}{Theorem}[section]
\newtheorem*{thm*}{Theorem}
\newtheorem{prop}{Proposition}[section]
\newtheorem*{prop*}{Proposition}
\newtheorem{cor}{Corollary}[section]
\newtheorem*{cor*}{Corollary}

\newtheorem*{lem*}{Lemma}

\theoremstyle{definition}

\newtheorem*{defn*}{Definition}

\newtheorem*{exmp*}{Example}
\newtheorem{exmps}{Examples}[section]
\newtheorem*{exmps*}{Examples}

\newtheorem{rem}{Remark}[section]
\newtheorem*{rem*}{Remark}
\newtheorem{rems}{Remarks}[section]
\newtheorem*{rems*}{Remarks}

\newtheorem*{note*}{Note}
\newcommand{\N}{{\mathbb N}}
\newcommand{\Z}{{\mathbb Z}}
\newcommand{\R}{{\mathbb R}}
\newcommand{\C}{{\mathbb C}}
\newcommand{\F}{{\mathbb F}}
\newcommand{\emps}{\emptyset}
\newcommand{\eps}{\varepsilon}
\renewcommand{\iff}{\: \Leftrightarrow\: }
\renewcommand{\bar}{\overline}

\numberwithin{equation}{section}
\DeclareMathOperator{\Rep}{Re\ignorespaces}
\DeclareMathOperator{\Imp}{Im\ignorespaces}

\DeclareMathOperator{\spa}{span}
\DeclareMathOperator{\orb}{orb}
 
\DeclareMathOperator{\Per}{Per} 
\begin{document}
\title[On conditions for linear hypercyclicity and chaos]
{On sufficient  and necessary conditions\\ for linear hypercyclicity and chaos}
\author[Marat V. Markin]{Marat V. Markin}
\address{
Department of Mathematics\newline
California State University, Fresno\newline
5245 N. Backer Avenue, M/S PB 108\newline
Fresno, CA 93740-8001
}
\email{mmarkin@csufresno.edu}
\subjclass{Primary 47A16, 47A10; Secondary 47B37, 47B38, 47A05}
\keywords{Hypercyclic vector, periodic point, hypercyclic operator, chaotic operator, spectrum}
\begin{abstract}
By strengthening one of the hypotheses of a well-known sufficient condition for the hypercyclicity of linear operators in Banach spaces, we arrive at a sufficient condition for linear chaos and reveal consequences of the latter for inverses, powers, multiples, and spectral properties. Extending the results, familiar for bounded linear operators, we also show that the hypercyclicity of unbounded linear operators subject to the sufficient condition for hypercyclicity is inherited by their bounded inverses, powers, and unimodular multiples and that necessary conditions for linear hypercyclicity stretch to the unbounded case.
\end{abstract}
\maketitle

\section[Introduction]{Introduction}

Prior to \cite{B-Ch-S2001,deL-E-G-E2003}, the notions of linear hypercyclicity and chaos had been studied exclusively for \textit{continuous} linear operators on Fr\'echet spaces, in particular for \textit{bounded} linear operators on Banach spaces (for a comprehensive survey, see \cite{Bayart-Matheron,Grosse-Erdmann-Manguillot}).

While entering the realm of unbounded linear hypercyclicity and chaos, paper \cite{B-Ch-S2001} provides a sufficient condition for hypercyclicity \cite[Theorem $2.1$]{B-Ch-S2001}, which has become a very useful shortcut for establishing hypercyclicity for (bounded or unbounded) linear operators without explicitly constructing hypercyclic vectors for them, employed in various works on the subject (see, e.g., \cite{B-Ch-S2001,Emam-Hesh2005,Intissar2014,arXiv:1912.03402,arXiv:1811.06640,arXiv:2106.09682}). 

By strengthening one of the hypotheses of \cite[Theorem $2.1$]{B-Ch-S2001}, we arrive at a sufficient condition for linear chaos and reveal consequences of the latter for inverses, powers, multiples, and spectral properties.

Extending the results, familiar for bounded linear operators, we also show that the hypercyclicity of unbounded linear operators subject to \cite[Theorem $2.1$]{B-Ch-S2001} is inherited by their bounded inverses, powers, and unimodular multiples and that necessary conditions for linear hypercyclicity stretch to the unbounded case.

While tackled in various works over the past twenty years (see, e.g., \cite{B-Ch-S2001,deL-E-G-E2003,Emam-Hesh2005,B-B-T2008,Intissar2014,arXiv:1912.03402,arXiv:2106.09682,Markin2020(1),arXiv:1811.06640,Mark-Sich2019(1)}), the study of unbounded linear hypercyclicity and chaos does not appear to have attained the maturity level warranting emergence of a survey in the manner of the aforementioned monographs \cite{Bayart-Matheron,Grosse-Erdmann-Manguillot}. This paper takes a step towards summarizing certain general facts on the subject.

\section[Preliminaries]{Preliminaries}

The following preliminaries are essential for our discourse.

\subsection{Linear Hypercyclicity and Chaos}\

For a (bounded or unbounded) linear operator $A$ in a (real or complex) Banach space $X$, a nonzero vector 
\begin{equation*}
f\in C^\infty(A):=\bigcap_{n=0}^{\infty}D(A^n)
\end{equation*}
($D(\cdot)$ is the \textit{domain} of an operator, $A^0:=I$, $I$ is the \textit{identity operator} on $X$) is called \textit{hypercyclic} if its \textit{orbit} under $A$
\[
\orb(f,A):=\left\{A^nf\right\}_{n\in\Z_+}
\]
($\Z_+:=\left\{0,1,2,\dots\right\}$ is the set of \textit{nonnegative integers}) is dense in $X$.

Linear operators possessing hypercyclic vectors are said to be \textit{hypercyclic}.

If there exist an $N\in \N$ ($\N:=\left\{1,2,\dots\right\}$ is the set of \textit{natural numbers}) and a vector 
\[
f\in D\left(A^N\right)\quad \text{with}\quad A^Nf = f,
\]
such a vector is called a \textit{periodic point} for the operator $A$ of period $N$. If $f\ne 0$, we say that $N$ is a \textit{period} for $A$.

Hypercyclic linear operators with a dense in $X$ set $\Per(A)$ of periodic points are said to be \textit{chaotic}.

See \cite{Devaney,Godefroy-Shapiro1991,B-Ch-S2001}.

\begin{rems}\label{HCrems}\
\begin{itemize}
\item In the prior definition of hypercyclicity, the underlying space is necessarily
\textit{infinite-dimensional} and \textit{separable} (see, e.g., \cite{Grosse-Erdmann-Manguillot}).
\item For a hypercyclic linear operator $A$, the set $HC(A)$ of its hypercyclic vectors is necessarily dense in $X$, and hence, the more so, is the subspace $C^\infty(A)\supseteq HC(A)$.
\item Observe that
\[
\Per(A)=\bigcup_{N=1}^\infty \Per_N(A),
\]
where 
\[
\Per_N(A)=\ker(A^N-I),\ N\in \N
\]
is the \textit{subspace} of $N$-periodic points of $A$.
\item As immediately follows from the inclusions
\begin{equation*}
HC(A^n)\subseteq HC(A),\ \Per(A^n)\subseteq \Per(A),\ n\in \N,
\end{equation*}
if, for a linear operator $A$ in an infinite-dimensional separable Banach space $X$ and some $n\ge 2$, the operator $A^n$ is hypercyclic or chaotic, then $A$ is also hypercyclic or chaotic, respectively.
\end{itemize} 
\end{rems}

\subsection{Resolvent Set and Spectrum}\

For a closed linear operator $A$ in a complex Banach space $X$, the set
\[
\rho(A):=\left\{ \lambda\in \C \,\middle|\, \exists\, {(A-\lambda I)}^{-1}\in L(X) \right\}
\]
($L(X)$ is the space of bounded linear operators on $X$) and its complement $\sigma(A):=\C\setminus \rho(A)$ are called the operator's \textit{resolvent set} and \textit{spectrum}, respectively (see, e.g., \cite{Markin2020EOT,Dun-SchI}).

The spectrum is a \textit{closed set} in $\C$, which is partitioned into three pairwise disjoint subsets, $\sigma_p(A)$, $\sigma_c(A)$, and $\sigma_r(A)$, called the \textit{point}, \textit{continuous}, and \textit{residual spectrum} of $A$, respectively, as follows:
\begin{equation*}
\begin{split}
& \sigma_p(A):=\left\{\lambda\in \C \,\middle|\,A-\lambda I\ \text{is \textit{not injective}, i.e., $\lambda$ is an \textit{eigenvalue} of $A$} \right\},\\
& \sigma_c(A):=\left\{\lambda\in \C \,\middle|\,A-\lambda I\ \text{is \textit{injective},
\textit{not surjective}, and $\overline{R(A-\lambda I)}=X$} \right\},\\
& \sigma_r(A):=\left\{\lambda\in \C \,\middle|\,A-\lambda I\ \text{is \textit{injective} and $\overline{R(A-\lambda I)}\neq X$} \right\}
\end{split}
\end{equation*}
($R(\cdot)$ is the \textit{range} of an operator and $\overline{\cdot}$ is the \textit{closure} of a set) (see, e.g., \cite{Dun-SchI,Markin2020EOT}).

\begin{rem}\label{remPP}
For an $N\in \N$,
\[
\Per_N(A)=\ker(A^N-I)\neq \left\{0\right\},\ N\in \N\iff 1\in \sigma_p\left(A^N\right).
\]
\end{rem}
\section{Sufficient Conditions}

\subsection{A Sufficient Condition for Linear Hypercyclicity}\

The succeeding statement is an extension of \textit{Kitai's ctriterion} for bounded linear operators \cite{Kitai1982,Gethner-Shapiro1987} allows to establish hypercyclicity for (bounded or unbounded) linear operators without explicit construction of hypercyclic vectors for them.

\begin{thm}[Sufficient Condition for Linear Hypercyclicity {\cite[Theorem $2.1$]{B-Ch-S2001}}]\label{SCH}\ \\
Let $X$ be a  (real or complex) infinite-dimensional separable Banach space and $A$ be a densely defined linear operator in $X$ such that each power $A^{n}$ ($n\in\N$) is a closed operator. If there exists a set
\[
Y\subseteq C^\infty(A):=\bigcap_{n=1}^\infty D(A^n)
\]
dense in $X$ and a mapping $B:Y\to Y$ such that
\begin{enumerate}
\item $\forall\, f\in Y:\ ABf=f$ and
\item $\forall\, f\in Y:\ A^nf,B^nf\to 0,\ n\to \infty$,
\end{enumerate}
then the operator $A$ is hypercyclic.
\end{thm}

\begin{rem}\label{remSCH}
Following in the footsteps of \cite{Rolewicz}, the proof of {\cite[Theorem $2.1$]{B-Ch-S2001}} provides a method for construction of hypercyclic vectors, which defines a hypercyclic vector for $A$ via the series
\begin{equation}\label{HV}
f:=\sum_{k=1}^\infty B^{n(k)}f_k,
\end{equation}
where $\left\{f_k\right\}_{k\in \N}$ is a countable dense subset of $Y$ and $\left(n(k)\right)_{k\in \N}$ is a subsequence of natural numbers. Condition (2) secures convergence for the series in \eqref{HV} and condition (1) along with the closedness of the powers of $A$ assures the fact that
\[
f\in D(A^{n(k)})\ \text{and}\ A^{n(k)}f=\sum_{j=1}^{k-1}A^{n(k)-n(j)}f_j+f_k
+\sum_{j=k+1}^{k-1}B^{n(j)-n(k)}f_j,\ k\in \N.
\]

Furthermore, condition (2) implies
\[
\|A^{n(k)}f-f_k\|\to 0,\ k\to \infty,
\]
and hence, the \textit{denseness} of the hypercyclic vectors in $X$.
\end{rem}

\subsection{A Sufficient Condition for Linear Chaos}\

From observations that hypothesis (2) of the \textit{Sufficient Condition for Linear Hypercyclicity} (Theorem \ref{SCH}) is frequently satisfied excessively, which results in the chaoticity of the operator in question 
(see, e.g., \cite{B-Ch-S2001,arXiv:1811.06640,arXiv:2106.09682,arXiv:1912.03402}), we arrive at

\begin{thm}[Sufficient Condition for Linear Chaos]\label{SCC}\ \\
Let $(X,\|\cdot\|)$ be a  (real or complex) infinite-dimensional separable Banach space and $A$ be a densely defined linear operator in $X$ such that each power $A^{n}$ ($n\in\N$) is a closed operator. If there exists a set
\[
Y\subseteq C^\infty(A):=\bigcap_{n=1}^\infty D(A^n)
\]
dense in $X$ and a mapping $B:Y\to Y$ such that
\begin{enumerate}
\item $\forall\, f\in Y:\ ABf=f$ and
\item $\forall\, f\in Y\  \exists\, \alpha=\alpha(f)\in (0,1),\ \exists\, c=c(f,\alpha)>0\ \forall\, n\in \N:$
\begin{equation*}
\max\left(\|A^nf\|,\|B^nf\|\right)\le c\alpha^n,
\end{equation*}
or equivalently,
\begin{equation}\label{(2(b))}
\forall\, f\in Y:\ \max\left(r(A,f),r(B,f)\right)<1,
\end{equation}
where 
\[
r(A,f):=\limsup_{n\to \infty}{\|A^nf\|}^{1/n}\quad \text{and}\quad
r(B,f):=\limsup_{n\to \infty}{\|B^nf\|}^{1/n},
\]
\end{enumerate}
then the operator $A$ is chaotic.
\end{thm}

\begin{proof}
Hypothesis (1) of the \textit{Sufficient Condition for Linear Chaos} (Theorem \ref{SCC}) replicating and its hypothesis (2) strengthening their respective counterparts in the \textit{Sufficient Condition for Linear Hypercyclicity} (Theorem \ref{SCH}). We infer by the latter that the operator $A$ is \textit{hypercyclic}.

Let $N\in \N$ and $f\in Y$ be arbitrary.

By hypothesis (2), the Laurent series
\begin{equation}\label{LS}
\sum_{m=-\infty}^{\infty}\lambda^{m}B^{mN}f
=\sum_{m=1}^{\infty}\lambda^{-m}A^{mN}f+f+\sum_{m=1}^{\infty}\lambda^{m}B^{mN}f
\end{equation}
($B^0:=I$ and $B^{-mN}:=A^{mN}$, $m\in \N$) converges for all $\lambda\in \F$ ($\F:=\R$ or $\F:=\C$) with
\begin{equation}\label{RC}
r\left(A^N,f\right)<|\lambda|<1/r\left(B^N,f\right)\quad (1/0:=\infty),
\end{equation}
where
\begin{equation}\label{rAN}
\begin{aligned}
&r\left(A^N,f\right)=\limsup_{m\to \infty}{\left\|A^{mN}f\right\|}^{1/m}=\limsup_{m\to \infty}{\left[{\left\|A^{mN}f\right\|}^{1/(mN)}\right]}^N\\
&\le\limsup_{m\to \infty}{\left[{\left\|A^{m}f\right\|}^{1/m}\right]}^N={r(A,f)}^N<1\\
&\text{and}\\
&r\left(B^N,f\right)=\limsup_{m\to \infty}{\left\|B^{mN}f\right\|}^{1/m}=\limsup_{m\to \infty}{\left[{\left\|B^{mN}f\right\|}^{1/(mN)}\right]}^N\\
&\le\limsup_{m\to \infty}{\left[{\left\|B^{m}f\right\|}^{1/m}\right]}^N={r(B,f)}^N<1,
\end{aligned}
\end{equation}
in particular for $\lambda=1$.

Thus, the vector
\begin{equation}\label{PP}
f_N:=\sum_{m=-\infty}^{\infty}B^{mN}f
=\sum_{m=1}^{\infty}A^{mN}f+f+\sum_{m=1}^{\infty}B^{mN}f
\in X
\end{equation}
is well defined.

Since, by hypothesis (1),
\[
\sum_{m=-\infty}^{\infty}A^{N}B^{mN}f=\sum_{m=-\infty}^{\infty}B^{(m-1)N}f=f_N,
\]
by the \textit{closedness} of the operator $A^{N}$, we infer that
\[
f_N\in D\left(A^{N}\right)\quad \text{and}\quad A^{N}f_N=f_N,
\]
(see, e.g., \cite{Markin2020EOT}), and hence, $f_N$ is an $N$-periodic point for $A$.

Further, by hypothesis (2)
\begin{equation}\label{EEN}
\begin{aligned}
&\exists\, \alpha=\alpha(f)\in (0,1),\ \exists\, c=c(f,\alpha)>0\ \forall\, m\in \N:\\
&\max\left(\|A^{mN}f\|,\|B^{mN}f\|\right)\le c\alpha^{mN}=c{\left(\alpha^N\right)}^m,
\end{aligned}
\end{equation}
where $0<\alpha^N\le \alpha<1$, and hence,
\begin{align*}
\left\|f_N-f\right\|
&=\left\|\sum_{m=1}^\infty A^{mN}+\sum_{m=1}^\infty B^{mN}f\right\|
\le \sum_{m=1}^\infty  \left\|A^{mN}f\right\|+\sum_{m=1}^\infty  \left\|B^{mN}f\right\|
\\
&\le 2c\sum_{m=1}^\infty {\left(\alpha^N\right)}^m=2c\frac{\alpha^{N}}{1-\alpha^{N}}\to 0,\ N\to \infty.
\end{align*}

In view of the denseness of $Y$ in $X$, we infer that the set $\Per(A)$ of periodic points of $A$ is also \textit{dense} in $X$, which implies that the operator $A$ is \textit{chaotic} and completes the proof.
\end{proof}

\begin{samepage}
\begin{rems}\label{SCCrems}\
\begin{itemize}
\item The proof of the \textit{Sufficient Condition for Linear Chaos} (Theorem \ref{SCC}) offers an approach to construction of periodic points, which defines an $N$-periodic point for $A$ via Laurent series \eqref{PP}, where $y\in Y$
and $N\in \N$ arbitrary.

Condition (2) secures convergence for the series in \eqref{PP} and condition (1) along with the closedness of the powers of $A$ assure the fact that
\[
f_N\in D\left(A^{N}\right)\quad \text{and}\quad A^{N}f_N=f_N.
\]

Furthermore, condition (2) implies
\[
\left\|f_N-f\right\|\to 0,\ N\to \infty,
\]
and hence, the \textit{denseness} of the periodic points in $X$.
\item If, in the \textit{Sufficient Condition for Linear Chaos} (Theorem \ref{SCC}) ,
\[
\exists\, M\in \N:\ \ker A^M\cap Y\neq \left\{ 0\right\},
\]
for any $N\ge M$, an $N$-periodic point for $A$, defined via Laurent series \eqref{PP} based on
a vector
\[
f\in \left(\ker A^M\cap Y\right)\setminus \left\{ 0\right\},
\]
in view of
\[
A^{mN}f=0,\ m\in \N,
\]
is actually given by the power series
\begin{equation}\label{PP1}
f_N:=\sum_{m=0}^{\infty}B^{mN}f.
\end{equation}
\item The operator $A$ in the \textit{Sufficient Condition for Linear Chaos} (Theorem \ref{SCC}) meets hypothesis (2) when
\[
Y\subseteq \bigcup_{n=1}^\infty \ker A^n,
\]
in which case,
\[
\forall\, f\in Y\ \exists\, M=M(f)\in \N\ \forall\, n\ge M:\ A^nf=0.
\]
\item The mapping $B:Y\to Y$ in the \textit{Sufficient Condition for Linear Chaos} (Theorem \ref{SCC}) meets hypothesis (2) when it is the restriction to $Y$ of a \textit{bounded linear operator} $B:X\to X$ with
\[
r(B)=\lim_{n\to \infty}{\|B^n\|}^{1/n}<1
\]
($r(\cdot)$ is the \textit{spectral radius} of an operator, here and henceforth, $\|\cdot\|$ also stands for the \textit{operator norm})
(see, e.g., \cite{Dun-SchI,Markin2020EOT}).

In this case,
\begin{equation*}
\begin{aligned}
\forall\,  f\in X:\ &\limsup_{n\to \infty}{\|B^nf\|}^{1/n}
\le \limsup_{n\to \infty}{\left(\|B^n\|\|f\|\right)}^{1/n} 
\\
&=\lim_{n\to \infty}{\|B^n\|}^{1/n} \lim_{n\to \infty}{\|f\|}^{1/n}\le r(B)<1
\end{aligned}
\end{equation*}

In particular, if $\|B\|<1$, by \textit{Gelfand's spectral radius theorem} (see, e.g., \cite{Markin2020EOT}),
\[
r(B)\le \|B\|<1,
\]
and hence,
\[
\begin{aligned}
&\exists\, \alpha:=\|B\|\in (0,1)\ \forall\, f\in X\ \exists\, c=c(f):=\|f\|+1>0\ \forall\, n\in \N:\\ 
&\|B^nf\|\le \|B^n\|\|f\|\le {\|B\|}^n\|f\|\le c\alpha^n.
\end{aligned}
\]

When the bounded linear operator $B:X\to X$ is \textit{quasinilpotent}, i.e.,
\begin{equation*}
r(B)=\lim_{n\to \infty}{\|B^n\|}^{1/n}=0,
\end{equation*}
(see, e.g., \cite{Markin2020EOT}), we have:
\begin{equation*}
\begin{aligned}
\forall\,  f\in X:\ &0\le \limsup_{n\to \infty}{\|B^nf\|}^{1/n}
\le \limsup_{n\to \infty}{\left(\|B^n\|\|f\|\right)}^{1/n} \\
&=\lim_{n\to \infty}{\|B^n\|}^{1/n} \lim_{n\to \infty}{\|f\|}^{1/n}=0<1.
\end{aligned}
\end{equation*}
\item For an unbounded linear operator $A$, verifying the closedness of all powers $A^n$ ($n\in \N$) in the \textit{Sufficient Condition for Linear Hypercyclicity} (Theorem \ref{SCH}) or the \textit{Sufficient Condition for Linear Chaos} (Theorem \ref{SCC}) may present more of a challenge than checking other hypotheses (cf. \cite{arXiv:2106.09682}).
\item The \textit{Sufficient Condition for Linear Chaos} (Theorem \ref{SCC}) allows to establish chaoticity for (bounded or unbounded) linear operators without explicit construction of both hypercyclic vectors and a dense set of periodic points for them (see \cite{arXiv:2106.09682}).
\end{itemize}
\end{rems}
\end{samepage}

\begin{exmps}\label{SCCexmps}\
\begin{enumerate}[label=\arabic*.]
\item Consistently with \cite{Rolewicz,Godefroy-Shapiro1991}, the \textit{Sufficient Condition for Linear Chaos} (Theorem \ref{SCC}) applies to the bounded weighted backward shifts
\[
A\left(x_k\right)_{k\in \N}:=w\left(x_{k+1}\right)_{k\in \N}\quad (|w|>1)
\]
in the (real or complex) sequence space $X:=l_p$ ($1\le p<\infty$) of $p$-\textit{summable sequences} or $X:=c_0$ of \textit{vanishing sequences}, the latter equipped with the supremum norm
\[
c_0\ni x:=(x_k)_{k\in \N}\mapsto \|x\|_\infty:=\sup_{k\in \N}|x_k|,
\] 
as well as, consistently with \cite{arXiv:1811.06640}, to their unbounded counterparts
\[
A\left(x_k\right)_{k\in \N}:=\left(w^kx_{k+1}\right)_{k\in \N}\quad (|w|>1)
\]
with maximal domain. 

In both cases, 
\begin{equation*}
Y:=c_{00}=\bigcup_{n=1}^\infty \ker A^n,
\end{equation*}
where $c_{00}$ is the \textit{dense} in $X$ subspace of \textit{eventually zero sequences} and
\begin{equation*}
\ker A^n=\spa\left(\left\{e_k\right\}_{1\le k\le n}\right),\ n\in \N,
\end{equation*}
with $e_n:=\left(\delta_{nk}\right)_{k\in \N}$ ($\delta_{nk}$ is the \textit{Kronecker delta}).

In the former case, the mapping $B:Y\to Y$ is the restriction to $Y$ of the \textit{bounded linear operator}
\[
X\ni (x_k)_{k\in \N}\mapsto B(x_k)_{k\in \N}:=w^{-1}\left(x_{k-1}\right)_{k\in \N}\in X \quad (x_0:=0),
\]
for which
\begin{equation}\label{RI1}
ABx=x,\ x\in X,
\end{equation}
and
\begin{equation}\label{BBS}
\|B\|={|w|}^{-1}<1
\end{equation}
(see, e.g., \cite{Markin2020EOT}).

In the latter case, the mapping $B:Y\to Y$ is the restriction to $Y$ of the \textit{quasinilpotent operator}
\[
X\ni (x_k)_{k\in \N}\mapsto B(x_k)_{k\in \N}:=\left(w^{-(k-1)}x_{k-1}\right)_{k\in \N}\in X\quad (x_0:=0),
\]
for which
\begin{equation}\label{RI2}
ABx=x,\ x\in X,
\end{equation}
and
\begin{equation*}
\|B^n\|\le \prod_{j=1}^n{|w|}^{-j}={|w|}^{-\frac{n(n+1)}{2}},\ n\in \N,
\end{equation*}
\cite{arXiv:1811.06640}, which implies that
\begin{equation}\label{UBS}
\lim_{n\to \infty}{\|B^n\|}^{1/n}=0.
\end{equation}
\item Consistently with \cite{arXiv:2106.09682}, the \textit{Sufficient Condition for Linear Chaos} (Theorem \ref{SCC}) also applies to 
the differentiation operator
\[
Df:=f'
\]
with maximal domain $D(D):=C^1[a,b]$ in the space $C[a,b]$ ($-\infty<a<b<\infty$) equipped with the maximum norm
\[
C[a,b]\ni f\mapsto \|f\|_\infty:=\max_{a\le x\le b}|f(x)|.
\]

In this case,
\begin{equation*}
Y:=P=\bigcup_{n=1}\ker D^n,
\end{equation*}
where $P$ is the dense in $C[a,b]$ subspace of \textit{polynomials} and 
\begin{equation*}
\ker D^n=\left\{ f\in P\,\middle|\,\deg f\le n-1  \right\},\ n\in \N,
\end{equation*}
and the mapping $B:Y\to Y$ is the restriction to $Y$ of the \textit{quasinilpotent} Volterra integration operator
\[
[Bf](x):=\int_a^x f(t)\,dt,\ f\in C[a,b],x\in [a,b],
\]
(see, e.g., \cite{Markin2020EOT}) for which
\begin{equation}\label{RI3}
ABf=f,\ f\in C[a,b],
\end{equation}
and 
\begin{equation}\label{D}
\lim_{n\to \infty}{\|B^n\|}^{1/n}=0.
\end{equation}
\item The subspace
\[
\Per_N(D)=\ker(D^N-I),\ N\in \N,
\]
of $N$-periodic points for the differentiation operator $D$ of the prior example (see Remarks \ref{HCrems}) is found from the differential equation
\[
f^{(n)}=f,
\]
and hence, provided the space $C[a,b]$ is \textit{complex}, is the $N$-dimensional subspace
\[
\spa\left(\left\{e^{\lambda_1 x},\dots,e^{\lambda_N x}\right\}\right)
\]
 of $C[a,b]$, where $\lambda_k$, $k=1,\dots,N$, are the distinct values of $\sqrt[N]{1}$. For the real space $C[a,b]$, the subspace $\Per_N(A)$ is also $N$-dimensional, its basis obtained by separating the real and imaginary parts of the foregoing exponentials, in view of the fact that essentially complex $N$th roots of $1$ occur in conjugate pairs $\lambda$, $\bar{\lambda}$, each pair contributing two exponential-trigonometric functions
\[
e^{\Rep\lambda x}\cos(\Imp\lambda x)\quad \text{and}\quad e^{\Rep\lambda x}\sin(\Imp\lambda x).
\] 

The following demonstrates the construction of $2$-periodic points for $D$.

For an arbitrary
\[
f\in \ker D^2\subset P=:Y,
\]
we have:
\[
f(x)=f(a)+f'(a)(x-a),\ x\in [a,b],
\]
and
\[
[B^{2m}f](x)=f(a)\frac{(x-a)^{2m}}{(2m)!}+f'(a)\frac{(x-a)^{2m+1}}{(2m+1)!},\ m\in \Z_+,
x\in [a,b],
\]

Hence, the corresponding $2$-periodic point for $D$ is
\begin{align*}
f_2(x)&=\left[\sum_{m=0}^{\infty}B^{2m}f\right](x)
=f(a)\sum_{m=0}^\infty\frac{(x-a)^{2m}}{(2m)!}
+f'(a)\sum_{m=0}^\infty\frac{(x-a)^{2m+1}}{(2m+1)!}\\
&=f(a)\cosh x+f'(a)\sinh x\\
&=f(a)\frac{e^{x-a}+e^{-(x-a)}}{2}+f'(a)\frac{e^{x-a}-e^{-(x-a)}}{2},\ x\in [a,b],
\end{align*}
which is consistent with the fact that
\[
\Per_2(D)=\ker(D^2-I)=\spa\left(\left\{e^{x},e^{-x}\right\}\right)
\]
(se Remarks \ref{HCrems}).
\end{enumerate}
\end{exmps}

\section{Necessary Conditions}

\subsection{Inverses, Powers, and Multiples}\

As is known \cite[Proposition $2.23$]{Grosse-Erdmann-Manguillot}, a bounded linear operator $A$ on a Banach space $X$, whose inverse $A^{-1}$ is a bounded linear operator on $X$, is hypercyclic \textit{iff} $A^{-1}$ is. The following generalization of the \textit{``only if''} part is a direct corollary of
the \textit{Sufficient Condition for Linear Hypercyclicity} (Theorem \ref{SCH}).

\begin{cor}[Hypercyclicity of Inverse]\label{HI}\ \\
If, for a hypercyclic linear operator $A$ in a  (real or complex) infinite-dimensional separable Banach space $X$ subject to the \textit{Sufficient Condition for Linear Hypercyclicity} (Theorem \ref{SCH}), there exist an inverse $A^{-1}$, which is a bounded linear operator on $X$, then $A^{-1}$ is hypercyclic.
\end{cor}

\begin{proof}
Suppose that a hypercyclic linear operator $A$ in a  (real or complex) infinite-dimensional separable Banach space $X$ is subject to the \textit{Sufficient Condition for Linear Hypercyclicity} (Theorem \ref{SCH}) and there exist an inverse $A^{-1}$, which is a bounded linear operator on $X$.

By hypothesis (1) of the \textit{Sufficient Condition for Linear Hypercyclicity} (Theorem \ref{SCH}),
\begin{equation}\label{IRI}
\forall\, f\in Y:\ A^{-1}f=A^{-1}(ABf)=(A^{-1}A)Bf=Bf,
\end{equation}
which implies that $B$ is the restriction of $A^{-1}$ to $Y$.

Further, for the dense set $Y$, since $Y\subseteq C^\infty (A)$,
\[
A:Y\to Y
\]
and, by hypotheses (1) and (2) of the \textit{Sufficient Condition for Linear Hypercyclicity} (Theorem \ref{SCH}), we have:
\begin{enumerate}
\item $\forall\, f\in Y:\ A^{-1}Af=f$ and
\item $\forall\, f\in Y$:
\[
{\left(A^{-1}\right)}^nf=B^{n}f\to 0,\ m\to\infty,
\]
and
\[
A^nf\to 0,\ m\to\infty.
\]
\end{enumerate}

Thus, with $A^{-1}$ and $A$ assuming the roles of $A$ and $B$, respectively, by the \textit{Sufficient Condition for Linear Hypercyclicity} (Theorem \ref{SCH}), the inverse operator $A^{-1}$ is \textit{hypercyclic}.
\end{proof}

\begin{cor}[Chaoticity of Inverse]\label{CI}\ \\
If, for a chaotic linear operator $A$ in a  (real or complex) infinite-dimensional separable Banach space $(X,\|\cdot\|)$ subject to the \textit{Sufficient Condition for Linear Hypercyclicity} (Theorem \ref{SCH}), there exist an inverse $A^{-1}$, which is a bounded linear operator on $X$, then $A^{-1}$ is chaotic.
\end{cor}

\begin{proof}
Suppose that a chaotic linear operator $A$ in a  (real or complex) infinite-dimensional separable Banach space  $(X,\|\cdot\|)$ is subject to the \textit{Sufficient Condition for Linear Hypercyclicity} (Theorem \ref{SCH}) and there exist an inverse $A^{-1}$, which is a bounded linear operator on $X$.

As is shown in the proof of the prior statement (see \eqref{IRI}), $B$ is the restriction of $A^{-1}$ to $Y$.

Further, for the dense set $Y$, since $Y\subseteq C^\infty (A)$,
\[
A:Y\to Y
\]
and, by hypotheses (1) and (2) of the \textit{Sufficient Condition for Linear Hypercyclicity} (Theorem \ref{SCH}), we have:
\begin{enumerate}
\item $\forall\, f\in Y:\ A^{-1}Af=f$ and
\item $\forall\, f\in Y\  \exists\, \alpha=\alpha(f)\in (0,1),\ \exists\, c=c(f,\alpha)>0\ \forall\, n\in \N:$
\begin{equation*}
\max\left(\left\|{\left(A^{-1}\right)}^nf\right\|,\|A^nf\|\right)=\max\left(\|B^nf\|,\|A^nf\|\right)\le c\alpha^n.
\end{equation*}
\end{enumerate}

Thus, with $A^{-1}$ and $A$ assuming the roles of $A$ and $B$, respectively,
by the \textit{Sufficient Condition for Linear Chaos} (Theorem \ref{SCC}), the inverse operator $A^{-1}$ is \textit{\textit{chaotic}}.
\end{proof}

\begin{rem}
Provided the underlying space is complex, the existence of an inverse $A^{-1}$, which is a bounded linear operator on $X$, is equivalent to $0\in \rho(A)$ (see Preliminaries).
\end{rem}


As follows from \cite[Theorem 1]{Ansari1995} and 
\cite[Corollary 3]{Leon-Saavedra-Muller2004}, respectively, for a bounded linear hypercyclic operator $A$ on a Banach space, its every power $A^n$ ($n\in \N$) and unimodular multiple $\lambda A$ ($|\lambda|=1$) are hypercyclic as well. These conclusions can be easily extended to the unbounded linear hypercyclic operators subject to the \textit{Sufficient Condition for Linear Hypercyclicity} (Theorem \ref{SCH}). The following two statements are direct corollaries of the latter.

\begin{cor}[Hypercyclicity of Powers]\label{HP}\ \\
For a hypercyclic linear operator $A$ in a  (real or complex) infinite-dimensional separable Banach space $X$ subject to the \textit{Sufficient Condition for Linear Hypercyclicity} (Theorem \ref{SCH}), each power $A^n$ ($n\in \N$) is hypercyclic.
\end{cor}

\begin{proof}
If a hypercyclic linear operator $A$ in a  (real or complex) infinite-dimensional separable Banach space $X$ is subject to the \textit{Sufficient Condition for Linear Hypercyclicity} (Theorem \ref{SCH}), then so is its power $A^{n}$ for each $n\in\N$. 

Indeed, let $n\in \N$ be arbitrary. Then, for the dense set $Y$, 
\[
Y\subseteq C^\infty(A)=C^\infty(A^n)
\]
and the mapping $B^n:Y\to Y$. 

Further, by conditions (1) and (2) of the \textit{Sufficient Condition for Linear Hypercyclicity} (Theorem \ref{SCH}), we have:
\begin{enumerate}
\item $\forall\, f\in Y:$
\[
A^nB^nf=A^{n-1}(AB^nf)=A^{n-1}((AB)B^{n-1}f)=A^{n-1}B^{n-1}f\\
=\dots=ABf=f
\] 
and
\item $\forall\, f\in Y$:
\[
{\left(A^n\right)}^mf=A^{mn}f\to 0,\ m\to\infty,
\]
and
\[
{\left(B^n\right)}^mf=B^{mn}f\to 0,\ m\to\infty.
\]
\end{enumerate}

Hence, by the \textit{Sufficient Condition for Linear Hypercyclicity} (Theorem \ref{SCH}), each power $A^{n}$ ($n\in\N$) is \textit{hypercyclic}.
\end{proof}

\begin{cor}[Hypercyclicity of Multiples]\label{HM}\ \\
For a hypercyclic linear operator $A$ in a  (real or complex) infinite-dimensional separable Banach space $(X,\|\cdot\|)$ subject to the \textit{Sufficient Condition for Linear Hypercyclicity} (Theorem \ref{SCH}), each unimodular multiple $\lambda A$ ($|\lambda|=1$) is hypercyclic.
\end{cor}

\begin{proof}
If a hypercyclic linear operator $A$ in a  (real or complex) infinite-dimensional separable Banach space $(X,\|\cdot\|)$ is subject to the \textit{Sufficient Condition for Linear Hypercyclicity} (Theorem \ref{SCH}), then so is its multiple $\lambda A$ for each $\lambda\in \F$ with $|\lambda|=1$.

Indeed, let $\lambda\in \F$ with $|\lambda|=1$ be arbitrary. Then, for the dense set $Y$, 
\[
Y\subseteq C^\infty(A)=C^\infty(A^n)
\]
and the mapping $\lambda^{-1}B:Y\to Y$.

Further, by conditions (1) and (2) of the \textit{Sufficient Condition for Linear Hypercyclicity} (Theorem \ref{SCH}), we have:
\begin{enumerate}
\item $\forall\, f\in Y:\ (\lambda A)(\lambda^{-1}B)f=(\lambda\lambda^{-1})ABf=ABf=f$ and
\item $\forall\, f\in Y$:
\[
\left\|{(\lambda A)}^nf\right\|={\left|\lambda\right|}^n\|A^nf\|=\|A^nf\|\to 0,\ n\to \infty,
\]
and
\[
\left\|{(\lambda^{-1}B)}^nf\right\|={\left|\lambda^{-1}\right|}^n\|B^nf\|=\|B^nf\|\to 0,\ n\to \infty.
\]
\end{enumerate}

Thus, by the \textit{Sufficient Condition for Linear Hypercyclicity} (Theorem \ref{SCH}), each unimodular multiple $\lambda A$ ($|\lambda|=1$) is \textit{hypercyclic}.
\end{proof}

The subsequent counterparts of the two prior statements are immediate implications of the \textit{Sufficient Condition for Linear Chaos} (Theorem \ref{SCC}).

\begin{cor}[Chaoticity of Powers]\label{CP}\ \\
For a chaotic linear operator $A$ in a  (real or complex) infinite-dimensional separable Banach space $(X,\|\cdot\|)$ subject to the \textit{Sufficient Condition for Linear Chaos} (Theorem \ref{SCC}), each power $A^n$ ($n\in \N$) is chaotic.
\end{cor}

\begin{proof}
If a chaotic linear operator $A$ in a  (real or complex) infinite-dimensional separable Banach space $(X,\|\cdot\|)$ is subject to the  \textit{Sufficient Condition for Linear Chaos} (Theorem \ref{SCC}), then so is each power $A^n$ then so is its power $A^{n}$ for each $n\in\N$. 

Indeed, let $n\in \N$ be arbitrary. Then, for the dense set $Y$, 
\[
Y\subseteq C^\infty(A)=C^\infty(A^n)
\]
and the mapping $B^n:Y\to Y$. 

Further, by conditions (1) and (2) of the \textit{Sufficient Condition for Linear Chaos} (Theorem \ref{SCC}), we have:
\begin{enumerate}
\item $\forall\, f\in Y:\ A^nB^nf=f$ (see the proof of Corollary \ref{HP}) and
\item $\forall\, f\in Y\ \exists\, \alpha=\alpha(f)\in (0,1),\ \exists\, c=c(\alpha)>0\ \forall\, m\in \N$:
\[
\max\left(\left\|{\left(A^n\right)}^mf\right\|,\left\|{\left(B^n\right)}^mf\right\|\right)
=\max\left(\|A^{mn}f\|,\|B^{mn}f\|\right)\le c\alpha^{mn}=c{\left(\alpha^n\right)}^m,
\]
where $0<\alpha^n\le \alpha<1$.
\end{enumerate}

Thus, by the  \textit{Sufficient Condition for Linear Chaos} (Theorem \ref{SCC}), each power $A^n$ ($n\in\N$) is \textit{chaotic}.
\end{proof}

\begin{cor}[Chaoticity of Multiples]\label{CM}\ \\
For a chaotic linear operator $A$ in a  (real or complex) infinite-dimensional separable Banach space $(X,\|\cdot\|)$ subject to the \textit{Sufficient Condition for Linear Chaos} (Theorem \ref{SCC}), each multiple $\lambda A$ ($|\lambda|\ge 1$) is chaotic.
\end{cor}

\begin{proof}
If a chaotic linear operator $A$ in a  (real or complex) infinite-dimensional separable Banach space $(X,\|\cdot\|)$ is subject to the \textit{Sufficient Condition for Linear Chaos} (Theorem \ref{SCC}),  then so is its multiple $\lambda A$ for each $\lambda\in \F$ with $|\lambda|\ge 1$. 

Indeed, let $\lambda\in \F$ with $|\lambda|=1$ be arbitrary. Then, for the dense set $Y$, 
\[
Y\subseteq C^\infty(A)=C^\infty(A^n)
\]
and the mapping $\lambda^{-1}B:Y\to Y$.

Further, by conditions (1) and (2), we have:
\begin{enumerate}
\item $\forall\, f\in Y:\ (\lambda A)(\lambda^{-1}B)f=(\lambda\lambda^{-1})ABf=ABf=f$ and
\item $\forall\, f\in Y\ \exists\, \alpha=\alpha(f)\in (0,1),\ \exists\, c=c(\alpha)>0\ \forall\, n\in \N$:
\begin{equation*}
\begin{aligned}
\max\left(\left\|{(\lambda^{-1}A)}^nf\right\|,\left\|{(\lambda^{-1}B)}^nf\right\|\right)
&={\left|\lambda^{-1}\right|}^n
\max\left(\left\|A^nf\right\|,\left\|B^nf\right\|\right)
\\
&\le {\left|\lambda^{-1}\right|}^nc\alpha^n=c{\left(\left|\lambda^{-1}\right|\alpha\right)}^n,
\end{aligned}
\end{equation*}
where $0<\left|\lambda^{-1}\right|\alpha\le \alpha<1$.
\end{enumerate}

Thus, by the \textit{Sufficient Condition for Linear Chaos} (Theorem \ref{SCC}), each multiple $\lambda A$ ($|\lambda|\ge 1$) is \textit{chaotic}.
\end{proof}

\subsection{Spectral Properties}\

The following statement  transfers certain conditions necessary for the hypercyclicity of bounded linear operators on Banach spaces, summarized in \cite{Bayart-Matheron,Grosse-Erdmann-Manguillot}, to their unbounded counterparts.

\begin{prop}[Necessary Conditions for Linear Hypercyclicity]\label{NCH}\ \\
If a densely defined closed linear operator $A$ in an infinite-dimensional separable Banach space $(X,\|\cdot\|)$ is hypercyclic, then each of the following statements holds.
\begin{enumerate}
\item For each 
$g^*\in C^\infty(A^*):=\bigcap_{n=1}^{\infty}D\left({(A^*)}^n\right)\setminus \left\{0\right\}$, its orbit $\orb(g^*,A^*)$ under $A^*$
under the \textit{adjoint} operator $A^*$ is unbounded.
\item The adjoint operator $A^*$ has no eigenvalues, i.e., 
$\sigma_p(A^*)=\emps$ provided the underlying space is complex.
\item $\forall\, \lambda\in \F:\ \overline{R(A-\lambda I)}=X$.
\item $\sigma_r(A)=\emps$ provided the underlying space is complex.
\end{enumerate} 
\end{prop}

\begin{proof}
Suppose that a densely defined closed linear operator in an infinite-dimen\-sional separable Banach space $(X,\|\cdot\|)$ is hypercyclic is \textit{hypercyclic}.

The \textit{adjoint operator} $A^*$, acting in the \textit{dual space} $X^*$, is well defined since $A$ is \textit{densely defined} (see Remarks \ref{HCrems}).

Our proof of part (1) \textit{by contradiction} extends that of \cite[Proposition $5.1$ (ii)]{Grosse-Erdmann-Manguillot} to the unbounded case.

Assume that there exists a $g^*\in C^\infty(A^*)\setminus \left\{0\right\}$, for which
\[
\orb(g^*,A^*):=\left\{{(A^*)}^ng^*\right\}_{n\in\Z_+}
\]
is \textit{bounded}, i.e.,
\[
\sup_{n\in \Z_+}\left\|{(A^*)}^ng^*\right\|<\infty.
\]

Let $f\in C^\infty(A)\setminus \left\{0\right\}$ be a \textit{hypercyclic vector} for $A$. Then
\[
\forall\, n\in \Z_+:\  \langle A^nf,g^*\rangle =\langle A^{n-1}f,A^*g^* \rangle=\dots
=\langle f,{(A^*)}^ng^*\rangle
\]
($\langle \cdot,\cdot \rangle$ is the \textit{pairing} between $X$ and $X^*$, $A^0:=I$ and ${(A^*)}^0:=I$, here and henceforth, the same symbol $I$ is used to designate the \textit{identity operator} in both $X$ and $X^*$), and hence,
\[
\sup_{n\in \Z_+}\left|\langle A^nf,g^*\rangle\right|
=\sup_{n\in \Z_+}\left|\langle f,{(A^*)}^ng^*\rangle\rangle\right|
\le \sup_{n\in \Z_+}\left\|{(A^*)}^ng^*\right\|\|f\|<\infty.
\]

The latter contradicts the fact that, in view of the hypercyclicity of $f$ and $g^*\neq 0$, the set
\[
\left\{\langle A^nf,g^* \rangle\right\}_{n\in \Z}
\]
is \textit{dense} in $\F$, proving part (1).

Our proof of part (2) \textit{by contradiction} extends those of \cite[Proposition $1.17$]{Bayart-Matheron} and \cite[Lemma $2.53$ (a)]{Grosse-Erdmann-Manguillot} to the unbounded case and that of 
\cite[Lemma 1]{Mark-Sich2019(1)} to a Banach space setting.

Assume that the adjoint operator $A^*$ has an eigenvalue $\lambda\in \F$, i.e.,
\[
\exists\, g\in X^*\setminus \left\{0\right\}:\ A^*g=\lambda g,
\]
and hence,
\[
g^*\in C^\infty(A^*)\setminus \left\{0\right\} \quad \text{and}\quad
\forall\, n\in \N:\ {(A^*)}^ng^*=\lambda^n g^*.
\]

Let $f\in C^\infty(A)\setminus \left\{0\right\}$ be a \textit{hypercyclic vector} for $A$. Then
\[
\begin{aligned}
\forall\, n\in \Z_+:\  \langle A^nf,g^*\rangle &=\langle A^{n-1}f,A^*g^* \rangle=\dots
=\langle f,{(A^*)}^ng^*\rangle=\langle f,\lambda^n g^*\rangle 
\\
&=\lambda^n\langle f,g^* \rangle \quad (0^0:=1).
\end{aligned}
\]

The latter contradicts the fact that, in view of the hypercyclicity of $f$ and $g^*\neq 0$, the set
\[
\left\{\langle A^nf,g^* \rangle\right\}_{n\in \Z}
\]
is \textit{dense} in $\F$, proving part (2).

The equivalence
\[
(2) \iff (3)
\]
follows the facts that, for an arbitrary $\lambda \in \F$,
\[
{(A-\lambda I)}^*=A^*-\lambda I
\]
and, by the \cite[Theorem II.$3.7$]{Goldberg},
\[
\ker(A^*-\lambda I)
=\ker\left({(A-\lambda I)}^*\right)=\left\{0\right\}\iff \overline{R(A-\lambda I)}=X.
\]

The implication
\[
(3) \Rightarrow (4)
\]
instantly follows from the definition of \textit{residual spectrum} (see Preliminaries).
\end{proof}

Converting the \textit{Necessary Conditions for Linear Hypercyclicity} (Proposition \ref{NCH}) into eq\-uivalent contrapositive, we obtain the subsequent 

\begin{cor}[Non-Hypercyclicity Test]\label{NT}\ \\
A densely defined closed linear operator $A$ in an infinite-dimensional separable Banach space $(X,\|\cdot\|)$ is non-hypercyclic if any of the following statements
holds.
\begin{enumerate}
\item There exists a 
$g^*\in C^\infty(A^*):=\bigcap_{n=0}^{\infty}D\left({(A^*)}^n\right)\setminus \left\{0\right\}$ such that its orbit $\orb(g^*,A^*)$ under $A^*$ is bounded.
\item The adjoint operator $A^*$ has eigenvalues, i.e., 
$\sigma_p(A^*)\ne \emps$ provided the underlying space is complex.
\item $\exists\, \lambda\in \F:\ \overline{R(A-\lambda I)}\ne X$.
\item $\sigma_r(A)\ne \emps$ provided the underlying space is complex.
\end{enumerate} 
\end{cor}

\begin{rem}
Parts (2)--(4) of the \textit{Necessary Conditions for Linear Hypercyclicity} (Proposition \ref{NCH}) and the \textit{Non-Hypercyclicity Test} (Corollary \ref{NT}) extend \cite[Lemma 1]{Mark-Sich2019(1)} and \cite[Proposition 1]{Mark-Sich2019(1)} to a Banach space setting. 
\end{rem}

The \textit{Sufficient Condition for Linear Chaos} (Theorem \ref{SCC}) has 
essential spectral outcomes, summarized in the next statement.

\begin{thm}[Spectral Properties]\label{SP}\ \\
Let $A$ be a chaotic linear operator in a  complex infinite-dimensional separable Banach space $(X,\|\cdot\|)$.
\begin{enumerate}
\item If $A$ is subject to the \textit{Sufficient Condition for Linear Chaos} (Theorem \ref{SCC}), then
\begin{equation}\label{PSI1}
\begin{aligned}
&\forall\, f\in Y\setminus \left\{0\right\},\ 
\forall\, r\in \left(r(A,f),1\right),\, R\in  \left(1,1/r(B,f)\right)\
\exists\, M\in \N\ \forall\, n\ge M:\\
&\left\{\lambda\in \C \,\middle|\,  |\lambda|=1 \right\}\subset \left\{\lambda\in \C \,\middle|\,  r^n\le |\lambda|\le R^n \right\}\subseteq \sigma_p(A^n),
\end{aligned}
\end{equation}
and hence,
\begin{equation}\label{PSI11}
\begin{aligned}
&\forall\, f\in Y\setminus \left\{0\right\},\ \forall\, \lambda\in \C\ \text{with}\
{r(A,f)}<|\lambda|<1/{r(B,f)}\ \exists\, M\in \N\\
&\forall\, n\ge M:\ \lambda \in \sigma_p(A^n).
\end{aligned}
\end{equation}
\item If $A$ is subject to the \textit{Sufficient Condition for Linear Chaos} (Theorem \ref{SCC}) with the following stronger version of hypothesis (2):
\begin{enumerate}
\item[(2*)] $\exists\, \alpha\in (0,1)\ \forall\, f\in Y\ \exists\, c=c(f,\alpha)>0\ \forall\, n\in \N:$
\begin{equation*}
\max\left(\|A^nf\|,\|B^nf\|\right)\le c\alpha^n,
\end{equation*}
or equivalently,
\[
\exists\, \beta\in (0,1)\ \forall\,  f\in Y:\ \max\left(r(A,f),r(B,f)\right)\le \beta,
\]
\end{enumerate}
then
\begin{equation}\label{PSI2}
\forall\, \gamma\in (\alpha,1)\ \exists\, M\in \N\ \forall\, n\ge M:\
\left\{\lambda\in \C \,\middle|\,  \gamma^n\le |\lambda|\le 1/\gamma^n \right\}\subseteq \sigma_p(A^n),
\end{equation}
and hence,
\begin{equation}\label{PSI21}
\forall\, \lambda\in \C\ \text{with}\
\alpha<|\lambda|<1/\alpha\ \exists\, M\in \N\ \forall\, n\ge M:\
\lambda \in \sigma_p(A^n).
\end{equation}
\item If $A$ is subject to the \textit{Sufficient Condition for Linear Chaos} (Theorem \ref{SCC}) with the following stronger version of hypothesis (2*):
\begin{enumerate}
\item[(2**)] $\forall\, f\in Y,\ \forall\, \alpha\in (0,1)\  \exists\, c=c(f,\alpha)>0\ \forall\, n\in \N:$
\begin{equation*}
\max\left(\|A^nf\|,\|B^nf\|\right)\le c\alpha^n,
\end{equation*}
or equivalently,
\[
\forall\,  f\in Y:\  r(A,f)=r(B,f)=0,
\]
\end{enumerate}
then
\begin{equation}\label{PSI3}
\forall\, \gamma\in (0,1)\ \exists\, M\in \N\ \forall\, n\ge M:\
\left\{\lambda\in \C \,\middle|\,  \gamma^n\le |\lambda|\le 1/\gamma^n \right\}\subseteq \sigma_p(A^n),
\end{equation} 
and hence,
\begin{equation}\label{PSI31}
\forall\, \lambda\in \C\setminus \left\{0\right\}\ \exists\, M\in \N\ \forall\, n\ge M:\
\lambda \in \sigma_p(A^n).
\end{equation}
\item If $A$ is subject to the \textit{Sufficient Condition for Linear Chaos} (Theorem \ref{SCC}) with the following stronger version of hypothesis (2):
\begin{enumerate}
\item[(2')] 
\begin{enumerate}[label=(\alph*)]
\item $\forall\, f\in Y\ \exists\, N=N(f)\in \N\ \forall\, n\ge N:\ A^nf=0$,\\
or equivalently,
\begin{equation}\label{2'}
Y\subseteq \bigcup_{n=1}^\infty \ker A^n,
\end{equation}
and
\item $\forall\, f\in Y\  \exists\, \alpha=\alpha(f)\in (0,1),\ \exists\, c=c(f,\alpha)>0\ \forall\, n\in \N:\\ \|B^nf\|\le c\alpha^n$,
\end{enumerate}
\end{enumerate}
then
\begin{equation}\label{PSI1.1}
\begin{aligned}
&\forall\, f\in Y\setminus \left\{0\right\},\ 
\forall\, R\in  \left(1,1/r(B,f)\right)\
\exists\, M\in \N\ \forall\, n\ge M:\\
&\left\{\lambda\in \C \,\middle|\,  |\lambda|\le 1 \right\}\subset \left\{\lambda\in \C \,\middle|\,  |\lambda|\le R^n \right\}\subseteq \sigma_p(A^n),
\end{aligned}
\end{equation}
and hence,
\begin{equation}\label{PSI11.1}
\begin{aligned}
&\forall\, f\in Y\setminus \left\{0\right\},\ \forall\, \lambda\in \C\ \text{with}\
|\lambda|<1/{r(B,f)}\ \exists\, M\in \N\\
&\forall\, n\ge M:\ \lambda \in \sigma_p(A^n).
\end{aligned}
\end{equation}
\item If $A$ is subject to the \textit{Sufficient Condition for Linear Chaos} (Theorem \ref{SCC}) with the following stronger version of hypothesis (2):
\begin{enumerate}
\item[(2'*)] 
\begin{enumerate}[label=(\alph*)]
\item $\forall\, f\in Y\ \exists\, N=N(f)\in \N\ \forall\, n\ge N:\ A^nf=0$ and
\item $\exists\, \alpha\in (0,1)\ \forall\, f\in Y\ \exists\, c=c(f,\alpha)>0\ \forall\, n\in \N:\ \|B^nf\|\le c\alpha^n$,
\end{enumerate}
\end{enumerate}
then
\begin{equation}\label{PSI2.1}
\forall\, \gamma\in (\alpha,1)\ \exists\, M\in \N\ \forall\, n\ge M:\
\left\{\lambda\in \C \,\middle|\,  |\lambda|\le 1/\gamma^n \right\}\subseteq \sigma_p(A^n),
\end{equation}
and hence,
\begin{equation}\label{PSI21.1}
\forall\, \lambda\in \C\ \text{with}\ 
|\lambda|<1/\alpha\ \exists\, M\in \N\ \forall\, n\ge M:\ \lambda \in \sigma_p(A^n).
\end{equation}
\item If $A$ is subject to the \textit{Sufficient Condition for Linear Chaos} (Theorem \ref{SCC}) with the following stronger version of hypothesis (2):
\begin{enumerate}
\item[(2'**)] 
\begin{enumerate}[label=(\alph*)]
\item $\forall\, f\in Y\ \exists\, N=N(f)\in \N\ \forall\, n\ge N:\ A^nf=0$ and
\item $\forall\, f\in Y,\ \forall\, \alpha\in (0,1)\  \exists\, c=c(f,\alpha)>0\ \forall\, n\in \N:\ \|B^nf\|\le c\alpha^n$,
\end{enumerate}
\end{enumerate}
then
\begin{equation}\label{PSI3.1}
\forall\, \gamma\in (0,1)\ \exists\, M\in \N\ \forall\, n\ge M:\
\left\{\lambda\in \C \,\middle|\, |\lambda|\le 1/\gamma^n \right\}\subseteq \sigma_p(A^n),
\end{equation} 
and hence,
\begin{equation}\label{PSI31.1}
\forall\, \lambda\in \C\ \exists\, M\in \N\ \forall\, n\ge M:\
\lambda \in \sigma_p(A^n).
\end{equation}
\item If $A$ is subject to the \textit{Sufficient Condition for Linear Chaos} (Theorem \ref{SCC}), with 
\begin{enumerate}
\item the mapping $B:Y\to Y$ being the restriction to $Y$ of a bounded linear operator $B:X\to X$, for which hypothesis (1) stands for all $f\in X$, i.e.,
\begin{equation}\label{RI}
\forall\, f\in X:\ ABf=f,
\end{equation}
and 
\begin{equation}\label{kerAn}
\forall\, n\in \N:\ \ker A^n \cap R(B^n)=\left\{0\right\};
\end{equation}
\item hypotheses (2'(a)) (same as (2'*(a)) and (2'**(a))) and (2'(b)), or (2'*(b)), or (2'**(b)), respectively, standing for all $f\in X$,
\end{enumerate}
then 
\begin{enumerate}[label=(\roman*)]
\item $\ker A\neq \left\{0\right\}$;
\item inclusions 
\begin{equation}\label{PSI111}
\forall\, f\in \ker A\setminus \left\{0\right\},\ \forall\, n\in \N:\ 
C(f,n):=\left\{\lambda\in \C \,\middle|\,  |\lambda|<1/{r(B^n,f)} \right\}\subseteq \sigma_p(A^n),
\end{equation}
or inclusions
\begin{equation}\label{PSI211}
\forall\, n\in \N:\
C(\alpha,n):=\left\{\lambda\in \C \,\middle|\,  |\lambda|<1/\alpha^n \right\}\subseteq \sigma_p(A^n),
\end{equation}
or equalities
\begin{equation}\label{PSI311}
\forall\, n\in \N:\ \sigma(A^n)=\sigma_p(A^n)=\C,
\end{equation}
respectively, hold; and
\item 
\begin{equation}\label{dim}
\dim\ker(A^n-\lambda I)=\dim\ker A^n,
\end{equation}
the mapping
\begin{equation}\label{IS}
\ker A^n\ni f_{n,0}\mapsto f_{n,\lambda}:=\sum_{m=0}^\infty \lambda^mB^{mn}f_{n,0}\in \ker(A^n-\lambda I)\quad (0^0:=1)
\end{equation}
being an \textit{isomorphism} between $\ker A^n$ and $\ker(A^n-\lambda I)$, for all
$n\in \N$ and $\lambda\in C(f,n)$, or $\lambda \in C(\alpha,n)$, or $\lambda \in \C$, respectively.
\end{enumerate}
\end{enumerate}
\end{thm}

\begin{proof}\
\begin{enumerate}
\item Suppose that $A$ is subject to the \textit{Sufficient Condition for Linear Chaos} (Theorem \ref{SCC}).

Let $f\in Y\setminus \left\{0\right\}$ be arbitrary. 

By hypothesis (2) of the \textit{Sufficient Condition for Linear Chaos} (Theorem \ref{SCC}), the vector
\begin{equation}\label{EV}
\begin{aligned}
f_{n,\lambda}&:=\sum_{m=-\infty}^{\infty}\lambda^m B^{mn}f\\
&=\sum_{m=1}^{\infty}\lambda^{-m} A^{mn}f+f+\sum_{m=1}^{\infty}\lambda^m B^{mn}f\in X,
\end{aligned}
\end{equation}
($B^{-mn}:=A^{mn}$, $m\in \N$) is well defined for all $n\in \N$ and $\lambda\in C_n$, where
\begin{equation}\label{ACn}
A(f,n):=\left\{\lambda\in \C \,\middle|\,  r\left(A^n,f\right)<|\lambda|<1/r\left(B^n,f\right) \right\},
\end{equation}
where
\begin{equation}\label{rAn}
r\left(A^n,f\right)\le {r(A,f)}^n<1\quad \text{and}\quad r\left(B^n,f\right)\le {r(B,f)}^n<1
\end{equation}
(cf. \eqref{LS}--\eqref{rAN}).

Since, by hypothesis (1) of the \textit{Sufficient Condition for Linear Chaos} (Theorem \ref{SCC}), for all $n\in \N$ and $\lambda\in C_n$,
\[
\sum_{m=-\infty}^{\infty}A^{n}(\lambda^m B^{mn}f)=\sum_{m=-\infty}^{\infty}\lambda^mB^{(m-1)n}f
=\lambda \sum_{m=-\infty}^{\infty}\lambda^{m-1}B^{(m-1)n}=\lambda f_{n,\lambda},
\]
by the \textit{closedness} of the operator $A^{n}$, we infer that
\[
f_{n,\lambda}\in D(A^{n})\quad \text{and}\quad A^{n}f_{n,\lambda}=\lambda f_{n,\lambda},
\]
(see, e.g., \cite{Markin2020EOT}), i.e.,
\[
f_{n,\lambda}\in \ker(A^n-\lambda I).
\]

Let
\begin{equation}\label{param}
r\in \left(r(A,f),1\right)\quad \text{and}\quad R\in  \left(1,1/r(B,f)\right)
\end{equation}
be arbitrary and $\eps,\eps'\in (0,1)$ be such that
\[
r(A,f)<\eps r<r \quad \text{and}\quad R<R/\eps'<1/r(B,f).
\]

Then, for all $n\in \N$ and $\lambda\in C_{n}(r,R)$, where, in view \eqref{rAn},
\begin{equation}\label{CnrR}
A(f,n,r,R):=\left\{\lambda\in \C \,\middle|\, r^n\le |\lambda|\le R^n\right\}\subset A(f,n),
\end{equation}
since
\begin{equation}\label{est}
\begin{aligned}
&\exists\, c=c(f,r,R)\ \forall\, m\in \N:\\
&\|A^{m}f\|\le c(\eps r)^m\quad \text{and}\quad \|B^{m}f\|
\le c{\left(\eps' R^{-1}\right)}^{m},
\end{aligned}
\end{equation}
we have:
\[
\|A^{mn}f\||\lambda|^{-m}\le c {(\eps r)}^{mn}|\lambda|^{-m}
=c {\left(\eps^n\right)}^m{\left(r^n|\lambda|^{-1}\right)}^m
\le c{\left(\eps^{n}\right)}^m,\ m\in\N,
\]
and
\[
\|B^{mn}f\||\lambda|^{m}
\le c{\left(\eps' R^{-1}\right)}^{mn}|\lambda|^{m}
=c{\left[{\left(\eps'\right)}^n\right]}^{m}{\left(R^{-n}|\lambda|\right)}^m
\le c{\left[{\left(\eps'\right)}^n\right]}^{m},\ m\in\N,
\]
and hence,
\begin{equation*}
\begin{split}
\|f_{n,\lambda}-f\|&= \left\|\sum_{m=1}^{\infty}\lambda^{-m} A^{mn}f+\sum_{m=1}^{\infty}\lambda^m B^{mn}f\right\| 
\\
&\le \sum_{m=1}^{\infty}\|A^{mn}f\||\lambda|^{-m}
+\sum_{m=1}^{\infty}\|B^{mn}f\||\lambda|^{m}
\\
&\le c\sum_{m=1}^{\infty}{\left(\eps^{n}\right)}^m
+c\sum_{m=1}^{\infty}{\left[{\left(\eps'\right)}^n\right]}^{m}
\\
&= c\frac{\eps^{n}}{1-\eps^{n}}+c\frac{{\left(\eps'\right)}^n}{1-{\left(\eps'\right)}^n}\to 0,\ n\to\infty,
\end{split}
\end{equation*}
which, since $f\neq 0$, implies that
\begin{equation}\label{event}
\exists\, M=M(f,r,R)\in \N\ \forall\, n\ge M:\ f_{n,\lambda}\ne 0.
\end{equation}

Thus, inclusions \eqref{PSI1} hold, with $f_{n,\lambda}$, defined by \eqref{EV}, being an eigenvector of $A^n$ associated with $\lambda\in A_{n}(f,r,R)$
for all sufficiently large $n\in \N$.

Since
\begin{equation*}
\begin{aligned}
&\forall\, f\in Y\setminus \left\{0\right\},\ \forall\, \lambda\in \C\ \text{with}\
{r(A,f)}<|\lambda|<1/{r(B,f)}\\
&\exists\, r\in \left(r(A,f),1\right),\, R\in  \left(1,1/r(B,f)\right)\ \forall\, n\in \N:\ \lambda \in A(f,1,r,R)\subset A(f,n,r,R),
\end{aligned}
\end{equation*}
\eqref{PSI11} follows immediately.
\item  Suppose that $A$ is subject to the \textit{Sufficient Condition for Linear Chaos} (Theorem \ref{SCC}) with hypothesis (2*).

Then
\[
\forall\, f\in Y:\ r(A,f)\le \alpha\quad \text{and}\quad \alpha^{-1}\le \frac{1}{r(B,f)},
\]
and hence, by part (1), inclusions \eqref{PSI2} hold.

Since
\begin{equation*}
\begin{aligned}
&\forall\, \lambda\in \C\ \text{with}\
\alpha< |\lambda|< 1/\alpha\ \exists\, \gamma \in (\alpha,1)\ \forall\, n\in \N:\\
&\lambda \in \left\{\lambda\in \C \,\middle|\,  \gamma \le |\lambda|\le 1/\gamma \right\}\subset \left\{\lambda\in \C \,\middle|\,  \gamma^n\le |\lambda|\le 1/\gamma^n \right\},
\end{aligned}
\end{equation*}
\eqref{PSI21} follows immediately.
\item Part (3) follows from part (1) Suppose that $A$ is subject to the \textit{Sufficient Condition for Linear Chaos} (Theorem \ref{SCC}) with hypothesis (2*).

Since, under hypothesis (2**), by part (2), inclusions \eqref{PSI2} hold for all $\alpha\in (0,1)$, we infer that inclusion \eqref{PSI3} holds.

Since
\begin{equation*}
\begin{aligned}
&\forall\, \lambda\in \C\setminus \left\{0\right\}\exists\, \gamma \in (0,1)\ \forall\, n\in \N:\\
&\lambda \in \left\{\lambda\in \C \,\middle|\,  \gamma \le |\lambda|\le 1/\gamma \right\}\subset \left\{\lambda\in \C \,\middle|\,  \gamma^n\le |\lambda|\le 1/\gamma^n \right\},
\end{aligned}
\end{equation*}
\eqref{PSI31} follows immediately.
\item Under hypothesis (2'(a)) (same as (2'*(a)) and (2'**(a))), for arbitrary $f\in Y\setminus \left\{0\right\}$, 
\begin{equation}\label{An}
\exists\, N=N(f)\ \forall\, n\ge N:\ f\in \ker A^n,
\end{equation}
and hence, for all $n\ge N$, and 
\[
\lambda\in A(f,n):=\left\{\lambda\in \C \,\middle|\,  r\left(A^n,f\right)<|\lambda|<1/r\left(B^n,f\right) \right\}
\]
with
\[
r\left(A^n,f\right)={r(A,f)}^n=0
\]
(cf. \eqref{ACn}), the vector $f_{n,\lambda}\in \ker (A^n-\lambda I)$, defined via Laurent series \eqref{EV} based on $f$, is actually given by the power series
\begin{equation}\label{EV1}
f_{n,\lambda}:=\sum_{m=0}^{\infty}\lambda^m B^{mn}f
\end{equation}
(cf. Remarks \ref{SCCrems}), which also converges for $\lambda=0$.

Thus, proof for parts (4)--(6) are readily obtained by modifying those for parts (1)--(3), respectively, in which power series replace the Laurent series 
and circles centered at $0$ replace the corresponding annuli ($r=0$).
\item Suppose that $A$ is subject to the \textit{Sufficient Condition for Linear Chaos} (Theorem \ref{SCC}), with the mapping $B:Y\to Y$ being the restriction to $Y$ of a \textit{bounded linear operator} $B:X\to X$, for which \eqref{RI} and \eqref{kerAn} are valid and hypotheses (2'(a)) (same as (2'*(a)) and (2'**(a))) and (2'(b)), or (2'*(b)), or (2'**(b)), respectively, stand for all $f\in X$.

Then
\[
\ker A\neq \left\{0\right\}.
\]

Indeed, otherwise there exists $A^{-1}:R(A)\to D(A)$ (see, e.g., \cite{Markin2020EOT}) and since, by \eqref{RI}, $R(A)=X$ and
\[
\forall\, f\in X:\ A^{-1}f=A^{-1}(ABf)=(A^{-1}A)Bf=Bf,
\]
we infer that $A^{-1}=B$, and hence, $0\in \rho(A)$. This, since
\[
\forall\, n\in \N\ \exists\, {\left(A^n\right)}^{-1}={\left(A^{-1}\right)}^n,
\]
further implies that $0\in \rho(A^n)$ foe all $n\in \N$, which, by part (4), is a contradiction.

Due to the inclusions
\[
\ker A^n\subseteq \ker A^{n+1},\ n\in\N,
\]
we further infer that
\[
\ker A^n\neq \left\{0\right\},\ n\in\N.
\]

Let $n\in \N$ and $\lambda\in C(f,n)$ (see \eqref{PSI111}), or $\lambda\in C(\alpha,n)$ (see \eqref{PSI211}), or $\lambda\in \C$, respectively, be arbitrary.

Since $B:X\to X$, to generate a vector $f_{n,\lambda}\in \ker (A^n-\lambda I)$ via power series \eqref{EV1}, we can take an arbitrary 
\[
f\in \ker A\setminus \left\{0\right\}=\left(\bigcap_{n=1}^\infty \ker A^n\right)\setminus \left\{0\right\}.
\]

By the \textit{continuity} of the linear operator $B:X\to X$ (see, e.g., \cite{Markin2020EOT}) and in view of \eqref{kerAn},
\begin{equation}\label{neq0}
\begin{aligned}
f_{n,\lambda}&=\sum_{m=0}^{\infty}\lambda^m B^{mn}f
=f+\sum_{m=1}^{\infty}\lambda^m B^{mn}f
\\
&=f+\lambda B^n\left[\sum_{m=1}^{\infty}\lambda^{m-1}B^{(m-1)n}f\right]\neq 0.
\end{aligned}
\end{equation}

Hence, $\lambda\in \sigma_p(A^n)$ and $f_{n,\lambda}$ is an eigenvector of $A^n$ associated with $\lambda$. 

Thus, inclusions \eqref{PSI111}, or inclusions \eqref{PSI211}, or equalities \eqref{PSI311}, respectively, hold.

For any $f\in \ker A^n$ \eqref{EV1} defines a vector
\[
f_{n,\lambda}\in \ker(A^n-\lambda I)
\]
such that, by \eqref{neq0}, $f_{n,\lambda}\neq 0$ whenever $f\neq 0$.

Conversely, for an arbitrary $f_{n,\lambda}\in \ker(A^n-\lambda I)$,
\[
f:=f_{n,\lambda}-\lambda B^{n}f_{n,\lambda}\in \ker A^n.
\]

Indeed, by \eqref{RI},
\[
A^nf=A^nf_{n,\lambda}-\lambda A^nB^{n}f_{n,\lambda}
=\lambda f_{n,\lambda}-\lambda f_{n,\lambda}=0.
\]

Furthermore, 
\begin{align*}
\sum_{m=0}^\infty \lambda^mB^{mn}f&=\sum_{m=0}^\infty \lambda^mB^{mn}(f_{n,\lambda}-\lambda B^{n}f_{n,\lambda})\\
&=\sum_{m=0}^\infty \lambda^mB^{mn}f_{n,\lambda}-\sum_{m=0}^\infty \lambda^{m+1}\lambda B^{(m+1)n}f_{n,\lambda}=f_{n,\lambda}.
\end{align*}

Thus, by the \textit{linearity} of $B:X\to X$, the mapping defined by \eqref{IS} is an \textit{isomorphism} between $\ker A^n$ and $\ker(A^n-\lambda I)$, and hence \eqref{dim} holds (see, e.g., \cite{Markin2020EOT}).
\end{enumerate}
\end{proof}

\begin{rems}\label{SPrems}\
\begin{itemize}
\item The method for construction of eigenvectors for the powers of $A$, furnished in the proof of the prior statement, generalizes that for construction of periodic points from the proof of the \textit{Sufficient Condition for Linear Chaos} (Theorem \ref{SCC}) (see Remarks \ref{SCCrems}), the latter being the particular case of the former for $\lambda=1$. 
\item Condition (2'*(b)) in parts (5) and (7) is met when the mapping $B:Y\to Y$ is the restriction to $Y$ of a \textit{bounded linear operator} $B:X\to X$ with $\|B\|<1$,
in which case \eqref{PSI2.1} and \eqref{PSI21.1} or \eqref{PSI211}, respectively, hold for $\alpha:=\|B\|$ (see Remarks \ref{SCCrems}).
\item Condition (2'**(b)) in parts (6) and (7) is met when the mapping $B:Y\to Y$ is the restriction to $Y$ of a \textit{quasinilpotent operator} $B:X\to X$ (see Remarks \ref{SCCrems}).
\item Since the spectrum of a bounded linear operator is a nonempty compact subset of $\C$ (see, e.g., \cite{Dun-SchI,Markin2020EOT}), conditions of parts (3), (6), and (7) are met by an \textit{unbounded} chaotic operator $A$ only.
\end{itemize}
\end{rems}

\begin{exmps}\label{SPexmps}\
\begin{enumerate}[label=\arabic*.]
\item Part (7) of the \textit{Spectral Properties} (Theorem \ref{SP}) 
applies to the bounded weighted backward shifts
\[
A\left(x_k\right)_{k\in \N}:=w\left(x_{k+1}\right)_{k\in \N}\quad (|w|>1)
\]
in the complex sequence space $l_p$ ($1\le p<\infty$) or $c_0$ 
(see Examples \ref{SCCexmps})  due to \eqref{RI1}, \eqref{BBS}, and the fact that, since
\begin{equation}\label{kerincl1}
\ker A^n=\spa\left(\left\{e_k\right\}_{1\le k\le n}\right)\ \text{and}\ R(B^n)=\spa\left(\left\{e_k\right\}_{k\ge n+1}\right),\ n\in \N,
\end{equation}
we have:
\begin{equation}\label{kerincl2}
\ker A^n \cap R(B^n)=\left\{0\right\},\ n\in \N,
\end{equation}
which is consistent with
\[
\sigma_p(A)=\left\{\lambda\in \C\,\middle|\,|\lambda|<|w|\right\}
\]
and 
\[
\dim\ker(A-\lambda I)=1,\ |\lambda|<|w|,
\]
(see, e.g., \cite{Markin2020EOT}).
\item Consistently with \cite{arXiv:1811.06640}, part (7) of the \textit{Spectral Properties} (Theorem \ref{SP}) applies to the unbounded weighted backward shifts
\[
A\left(x_k\right)_{k\in \N}:=\left(w^kx_{k+1}\right)_{k\in \N}\quad (|w|>1)
\]
with maximal domain in the complex sequence space $l_p$ ($1\le p<\infty$) or $c_0$ (see Examples \ref{SCCexmps}) due to \eqref{RI2}, \eqref{UBS}, and the fact that \eqref{kerincl1} and \eqref{kerincl2} hold in this case as well.
\item Consistently with \cite{arXiv:2106.09682}, part (7) of the \textit{Spectral Properties} (Theorem \ref{SP}) applies to the differentiation operator
\[
Df:=f',
\]
with maximal domain $D(D):=C^1[a,b]$ in the complex space $C[a,b]$ ($-\infty<a<b<\infty$)
(see Examples \ref{SCCexmps}) due to \eqref{RI3}, \eqref{D}, and the fact that, since
\begin{equation*}
\ker D^n=\left\{ f\in P\,\middle|\,\deg f\le n-1  \right\},\ n\in \N,
\end{equation*}
and 
\[
R(B^n)=\left\{ f\in C^n[a,b]\,\middle|\, f^{(k)}(a)=0,\, k=1,\dots,n-1 \right\},\ n\in \N,
\]
we have:
\[
\ker A^n \cap R(B^n)=\left\{0\right\},\ n\in \N.
\]
\end{enumerate}
\end{exmps}

The ensuing corollary of the \textit{Spectral Properties} (Theorem \ref{SP}) is consistent with \cite[Proposition $5.7$]{Grosse-Erdmann-Manguillot} stipulating that the point spectrum of a bounded linear operator on a complex infinite-dimensional separable Banach space contains infinitely many roots of $1$.

\begin{cor}[Unit Circle/Disk]\label{UC}\ \\
For a chaotic linear operator in a  complex infinite-dimensional separable Banach space $X$ subject to the \textit{Sufficient Condition for Linear Chaos} (Theorem \ref{SCC}), 
\begin{equation*}
\exists\, M\in \N\ \forall\, n\ge M:\ \left\{\lambda\in \C \,\middle|\,  |\lambda|=1 \right\}\subset \sigma_p(A^n).
\end{equation*}

Furthermore, under the conditions of part (7) of the \textit{Spectral Properties} (Theorem \ref{SP}),
\begin{equation*}
\forall\, n\in \N:\ \left\{\lambda\in \C \,\middle|\,  |\lambda|\le 1 \right\}\subset \sigma_p(A^n).
\end{equation*}
\end{cor}

 
\end{document}